\documentclass{article}
\usepackage{amsthm,amsmath,amssymb,authblk}
\usepackage{amsthm}
\usepackage{hyperref}
\usepackage{pst-all}
\usepackage{enumerate}
\usepackage{enumitem}
\usepackage{tikz}
\usepackage[margin=1.4in]{geometry}
\usepackage[labelformat=simple]{subfig}
\newtheorem{thm}{Theorem}
\newtheorem{lem}[thm]{Lemma}
\newtheorem{prop}[thm]{Proposition}

\theoremstyle{definition}

\newcommand{\CCC}{\mathbb{C}}
\newcommand{\NNN}{\mathbb{N}}

\begin{document}
\title{A set of chromatic roots which is dense in the complex plane and closed under multiplication by positive integers}
\author{Adam Bohn}
\affil{Mathematics Research Institute\\ University of Exeter\\ Exeter, EX4 4QF\\ UK\\\tt a.s.bohn@exeter.ac.uk}

\maketitle

\begin{abstract}
We study a very large family of graphs, the members of which comprise disjoint paths of cliques with extremal cliques identified.  This broad characterisation naturally generalises those of various smaller families of graphs having well-known chromatic polynomials.  We derive a relatively simple formula for an arbitrary member of the subfamily consisting of those graphs whose constituent clique-paths have at least one trivial extremal clique, and use this formula to show that the set of all non-integer chromatic roots of these graphs is closed under multiplication by natural numbers.  A well-known result of Sokal then leads to our main result, which is that there exists a set of chromatic roots which is closed under positive integer multiplication in addition to being dense in the complex plane.  Our findings lend considerable weight to a conjecture of Cameron, who has suggested that this closure property may be a generic feature of the chromatic polynomial.  We also hope that the formula we provide will be of use to those computing with chromatic polynomials.
%
%
\end{abstract}
\section{Introduction}
Given some $q\in\NNN$ and a graph $G$ (all graphs will be assumed to be simple, that is without loops or multiple edges), a \emph{proper $q$-colouring} of $G$ is a function from the vertices of $G$ to a set of $q$ colours, with the property that adjacent vertices receive different colours.  The \emph{chromatic polynomial} $P_G(X)$ of $G$ is the unique monic polynomial which, when evaluated at $q$, gives the number of proper $q$-colourings of $G$.  The chromatic polynomial has been the subject of a huge amount of research since it was introduced over a century ago by Birkhoff \cite{Birkhoff:Determinant}.  For a comprehensive introduction to this polynomial see \cite{read:introduction}.

A \emph{chromatic root} of a graph is simply a zero of that graph's chromatic polynomial.  Often we only aim to know whether or not a given number $\alpha$ is a zero of \emph{some} chromatic polynomial, in which case we simply say that $\alpha$ is (not) a chromatic root, without specifying a graph.

Although a number of papers have been published on graphs having integral chromatic roots (\cite{Dong:Chordal},\cite{Luca:Integer}), these tend to be either uninteresting or else very rare, and the bulk of the relevant literature focuses  on non-integral zeros of chromatic polynomials.  These are somewhat mysterious, in that they are produced by a simple counting process and yet have no obvious combinatorial interpretation.  With certain notable exceptions, such as Sokal's discovery \cite{Sokal:Potts} that the maximum degree of a graph determines a bound on the absolute values of its chromatic roots, little progress has been made towards an understanding of the mechanism relating a graph's abstract structure to properties of its non-integer chromatic roots.

On the other hand, the application of analytic methods to this subject has proved resoundingly successful in determining the distribution of chromatic roots on the real line and in the complex plane.  The most important results in this area can be summarised by the following theorem,  the three parts of which are due, respectively, to Jackson \cite{Jackson:Interval}, Thomassen \cite{Thomassen:Intervals} and Sokal \cite{Sokal:Dense}:
\paragraph{Theorem}
The zero-free intervals and regions for the chromatic polynomial are entirely classified by the following 3 results:
\begin{enumerate}[label=(\emph{\roman{enumi})}]
\item There are no negative real chromatic roots, and none in either of the intervals $(0,1)$ or $(1,32/27]$
\item Chromatic roots are dense in the remainder of the real line
\item Chromatic roots are dense in the whole complex plane
\end{enumerate}

These results collectively represent quite a triumph in the study of chromatic polynomials: the results of Jackson and Thomassen provide a complete classification of those intervals of the real line in which zeros of chromatic polynomials do not occur, and Sokal's theorem---which the present work builds upon---states that there are no analogous ``forbidden regions'' for complex chromatic roots\footnote{It should be stressed that Sokal's findings do not preclude the existence of curves which contain no chromatic roots; in fact, it is widely suspected that there are no purely imaginary chromatic roots}.  However, they are of little help when it comes to determining whether or not a given number having no conjugates in forbidden intervals is a chromatic root of some graph.  As chromatic polynomials are monic with integer coefficients, their zeros are by definition algebraic integers; our current lack of insight into the converse question of which algebraic integers are chromatic roots is largely what motivates the current work.

The main result of this paper is the following extension of Sokal's result, which we shall prove by explicit construction.

\begin{thm}\label{thm:main}
There exists a set of chromatic roots which is dense in the complex plane, and closed under multiplication by natural numbers.
\end{thm}

It would be dubious to claim that this constitutes a strengthening of the previously stated analytic results.  After all, we know that the set of all algebraic integers is dense in $\CCC$, and that this set strictly contains the set of all chromatic roots; as density of sets is not a property which can be qualified, we cannot hope to ``improve'' on Sokal's theorem.  However, we are still very much lacking in knowledge of the algebraic properties of the chromatic polynomial\footnote{This is in sharp contrast to the algebraic \emph{methods} expounded by Biggs et al. \cite{Biggs:Algebraic},\cite{Biggs:Matrix}, which have  been used to great effect by Salas and Sokal \cite{Salas:Transfer} among others}, and Theorem \ref{thm:main} does contribute in this respect, as we shall explain.

At the time of writing only a handful of algebraic integers outside of the forbidden intervals have been definitively shown not to be chromatic roots (notable examples are certain Beraha constants---numbers of the form $4\cos^2(\pi/n), n\in \NNN$---and some generalized Fibonacci numbers).  To our knowledge not a single such number with non-zero imaginary part is known.  In the absence of any reliable means via which to rule out a number being a chromatic root, two conjectures were proposed at a Newton Institute workshop in 2008 which address this issue from a quite different perspective (these appear in an as-yet-unpublished manuscript of Cameron and Morgan's \cite{Cameron:Algebraic}).  The first of these, known as the ``$\alpha+n$ conjecture", essentially claims that chromatic polynomials are as ``algebraically diverse'' as can be, in that every number field is contained in the splitting field of some chromatic polynomial.  This has been proved for quadratic fields, initially by participants of the aforementioned workshop. Later a different technique was used and extended to cubic fields by the present author; in \cite{Bohn:Bicliques} an algorithm is given to construct infinite families of bipartite graphs having complements with the desired chromatic splitting field.

The second proposed conjecture (the ``$n\alpha$ conjecture''), which is strengthened by the work presented here, asserts that the set of all chromatic roots is closed under multiplication by positive integers.  Notwithstanding the results of this paper, empirical evidence would appear to be in favour of the general conjecture.  In particular a direct additive analogue follows from the well-known fact that if $\alpha$ is a chromatic root of a graph $G$, then $\alpha +n$ is a chromatic root of the join of $G$ with $K_n$ (this can be shown by a simple counting argument).   However, it is difficult to see a way to approach a full proof of this conjecture.  Theorem \ref{thm:main} is, to the best of our knowledge, the only real evidence to be found since the conjecture was proposed, and while it is certainly noteworthy, it does not entirely succeed in escaping the imprecision inherent in the analytical approach, a leap which will clearly need to be made if we are to further our algebraic understanding of the chromatic polynomial.

The remainder of this paper shall be dedicated to a proof of Theorem \ref{thm:main}.  In order to achieve this we will derive a general formula for a large class of graphs which generalise those used by Sokal to prove the density of chromatic roots in the complex plane.

\section{Clique-theta graphs}

We define a \emph{generalised theta graph} $\Theta_{m_1,\ldots,m_n}$ to be a graph consisting of two vertices joined by $n$ otherwise disjoint paths of lengths $m_1,\ldots,m_n$.  In the study of the location of chromatic roots, generalised theta graphs have been the subject of some interest (see for example \cite{Brown:Theta},\cite{Brown:Negative} and \cite{Hickman:Roots}). Most significantly, these were the graphs that Sokal used in his proof of the density result stated previously\footnote{In fact he was able to prove this result using the considerably smaller subfamily whose members have paths of uniform length.  Note that this restriction allows only for specification of two variables: the number of paths, and length of the paths in a given graph.  It seems little short of astounding that a polynomial furnishing a dense set of zeros upon specialisation into $\NNN$ of just two parameters could occur so naturally}.

With this in mind, it simplifies matters to approach our main result via the following more specialised proposition, which is essentially the $n\alpha$ conjecture restricted to generalised theta graphs.

\begin{prop}\label{prop:main}
If $\alpha$ is a non-integer chromatic root of a generalised theta graph, then $p\alpha$ is a chromatic root for all natural numbers $p$.
\end{prop}

We will need the following basic properties of the chromatic polynomial, the proof of which can be found in any text on the subject.

\begin{prop}\label{prop:basic}
Let $G$ be a simple graph.
\begin{enumerate}[label=(\roman{enumi})]
\item Let $G\backslash e$ be the graph obtained from $G$ by deleting the edge $e$, and let $G/e$ be obtained by contracting the endpoints of $e$ into a single vertex and deleting any resulting loops and all but one parallel edges.  Then
    \[P_G(X)=P_{G\backslash e}(X)-P_{G/e}(X).\]
    If $e\notin E(G)$, then we have:
    \[P_G(X)=P_{G+e}(X)+P_{G+e/e}(X),\]
    where $G+e$ is the graph is obtained by adding the edge $e$ to $G$.
\item If $G$ is a $K_n$-sum of two subgraphs $H_1$ and $H_2$ (that is, if $H_1 \cup H_2 = G$ and $H_1 \cap H_2 = K_n$), then
    \[P_G(X)=\frac{P_{H_1}(X)P_{H_2}(X)}{(X)_{n}}.\]
\end{enumerate}

\end{prop}
We will refer to the two identities in Proposition \ref{prop:basic} \emph{(i)} as, respectively, \emph{deletion-contraction} and \emph{addition-contraction}.

Our proof of Proposition \ref{prop:main} entails the examination of a further generalisation of theta graphs (as well as, incidentally, the graphs used to verify the quadratic case of the $\alpha +n$ conjecture), which we shall refer to as \emph{clique-theta graphs}.  We must first introduce some simpler graphs which we shall use as building blocks.

We define the \emph{clique-path} $L(a_1,\ldots,a_n)$ to be a path of length $n$, whose $i$th verteX has been replaced by a clique of size $a_i$, and whose edges have been replaced by all possible edges between neighbouring cliques. Proposition \ref{prop:basic} \emph{(ii)} enables us to quickly ascertain that the chromatic polynomial of this graph is of the form:

\[P_{L(a_1,\ldots,a_n)}(X)=\frac{(X)_{a_1+a_2}\ldots(X)_{a_{n-1}+a_n}}{(X)_{a_2}\ldots(X)_{a_{n-1}}}\]

Clique-paths are a subfamily of what were originally referred to by Read as \emph{clique-graphs}; that is graphs consisting of an underlying structure, whose vertices have been ``blown up'' into cliques.  A more interesting subfamily, introduced in \cite{read:family}, is obtained by using an underlying cycle structur: let $(a_1,\ldots,a_n)$ be an ordered $n$-tuple of natural numbers.  Then the \emph{ring of cliques} $R(a_1,\ldots,a_n)$ is defined to be a simple $n$-cycle, whose $i$th vertex has been replaced by a clique of size $a_i$, and whose edges have been replaced by all possible edges between neighbouring cliques.  In the paper cite above, Read gives the following formula for the chromatic polynomial of this graph:

\[P_{R(a_1,\ldots a_n)}(X)=(X)_{a_1+a_2}\ldots (X)_{a_n+a_1}\sum_{k=0}^n(-1)^{nk}v_k(X)\left(\prod_{i=1}^n\frac{-(a_i)_k}{(X)_{a_i+k}}\right),\]
where $v_k(X)={X \choose k}-{X \choose k-1}$.

Perhaps contrary to first impressions this is indeed a polynomial: the terms in the denominator of the summation are cancelled by some of the preceding linear factors.  Interestingly, a permutation of the $\{a_i\}$ may change the linear factors, but does not affect the final more complicated factor.

\begin{figure}[!htbp]
\begin{center}
\begin{tikzpicture}
  [scale=.4,auto=center,every node/.style={circle,fill=black!100,inner sep=1mm},
    every path/.style={draw,thick,black}]
  \node (n1) at (10,-2) {};
  \node (n2) at (3,3) {};
  \node (n3) at (3,5) {};
  \node (n4) at (10,8) {};
  \node (n5) at (8.5,10) {};
  \node (n6) at (11.5,10) {};
  \node (n7) at (17,5) {};
  \node (n8) at (17,3) {};

 \foreach \from/\to in {n1/n2,n1/n3,n2/n4,n2/n5,n2/n6,n3/n4,n3/n5,n3/n6,n4/n7,n5/n7,n6/n7,n4/n8,n5/n8,n6/n8,n7/n1,n8/n1}
    \path[bend left=15,thick] (\from) edge (\to);

\foreach \from/\to in {n7/n8,n2/n3}
     \path[bend left=10,thick] (\from) edge (\to);

\foreach \from/\to in {n4/n5,n5/n6}
     \path[bend left=5,thick] (\from) edge (\to);

\foreach \from/\to in {n4/n6}
     \path[bend right=5,thick] (\from) edge (\to);

\end{tikzpicture}
\caption{The ring of cliques $R(1,2,3,2)$}
\end{center}
\end{figure}
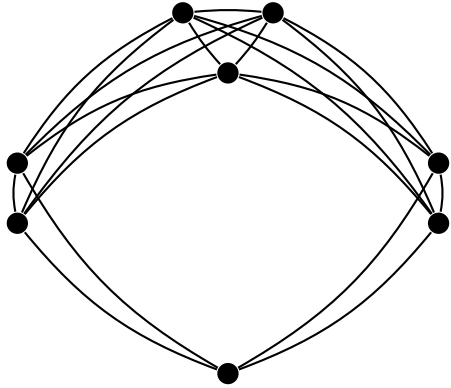

Chromatic polynomials of rings of cliques possess some intriguing properties.  For example, in \cite{Dong:Ring} it is shown that the real part of a non-integer chromatic root of a ring of four cliques is dependent only on the number of vertices in the graph.  As previously mentioned, rings of four cliques were also used in \cite{Cameron:Algebraic} to verify the quadratic case of the $\alpha + n$ conjecture.

In the case $a_1=1$, the chromatic polynomial $P_{R(1,a_2,\ldots a_n)}(X)$ reduces to the following considerably simpler expression (this was also discovered---but not published---by Read; a separate derivation is given in \cite{Dong:Chordal}):

\begin{equation}\label{eqn:cliquering}
X(X-1)_{a_{n-1}+a_n-1}\left(\prod_{i=2}^{n-2}(X-a_{i+1}-1)_{a_i-1}\right)r(1,a_2,\ldots,a_n;X),
\end{equation}
where
\[r(1,a_2,\ldots,a_n;X)=\frac{1}{X}\left(\prod_{i=2}^n(X-a_i)-\prod_{i=2}^n(-a_i)\right).\]
Note that the only (possibly) non-linear factor here is $r(1,a_2,\ldots,a_n;X);$ later we will refer to a factor of this form as the ``interesting factor'' of a ring of cliques.

We are now in a position of being able to compute the chromatic polynomials of the graphs used in the proof of Proposition \ref{prop:main}.  These are also clique-graphs, this time having underlying structures of generalised theta graphs; as such they are generalisations of both rings of cliques and theta graphs.

Formally, we define the \emph{clique-theta graph} $T(j,S_1,\ldots,S_n,k)$ in the following way: let $S_1,\ldots,S_n$ be $n$ non-empty ordered sets of positive integers with $S_i=(a_{i(1)},a_{i(2)},\ldots,a_{i(m_i)})$ for each $1\leq i\leq n$.  For each set $S_i$, let $L(j,S_i,k)$ be a clique path with extremal\footnote{To clarify, we use the term extremal here to refer to endpoints, as opposed to any notion of size} cliques of order $j$ and $k$, and clique sizes otherwise determined by the elements of $S_i$.  Then $T(j,S_1,\ldots,S_n,k)$ is the graph obtained by identifying the extremal cliques of the $\{L(j,S_i,k)\}_{1\leq i\leq n}$.

As might be expected from the complicated formula for general rings of cliques, the chromatic polynomials of these graphs are difficult to compute, and any general formula is likely to be almost impossible to express in any kind of meaningful way.  However, as with rings of cliques, we can considerably simplify the task of finding such a formula by specifying that $j=1$.  It is highly convenient that this restriction is in fact necessary to guarantee the algebraic properties we need to prove Proposition \ref{prop:main}!

\begin{prop}\label{prop:cliquetheta}
The chromatic polynomial of a clique-theta graph $T(1,S_1,\ldots,S_n,k)$ is:

\begin{align*}
&\left[(X)_{a_{n(m_n)}+k}\left(\prod_{i=1}^{n-1}(X-k-1)_{a_{i(m_i)}-1}\right)\left(\prod_{i=1}^n\prod_{l=1}^{m_i-1}(X-a_{i(l+1)}-1)_{a_{i(l)}-1}\right)\right] \\
&\times \left[\left(k(X-k)^{n-1}\prod_{i=1}^nr(1,a_{i(1)},\ldots,a_{i(m_i)};X)\right)+\left(\prod_{i=1}^nr(1,a_{i(1)},\ldots, a_{i(m_i)},k;X)\right)\right],
\end{align*}
where $r(1,a_{i(1)},\ldots,a_{i(m_i)};X)$ is the interesting factor of a ring of cliques\\ $R(1,a_{i(1)},\ldots,a_{i(m_i)})$.
\end{prop}

The proof of this result can be derived using standard identities for the chromatic polynomial, however the following lemma will greatly simplify its presentation.
\begin{lem}\label{lem:thetarecur}
Let $S_1=(a_{1(1)},\ldots,a_{1(m_1)})$, and let $\bar{S_1}=(a_{1(2)},\ldots,a_{1(m_1)}).$  Then:
\begin{align*}
P_{T(1,S_1,\ldots,S_n,k)}(X)&=\frac{P_{T(1,S_2,\ldots,S_n,k)}(X)P_{L(a_{1(1)},\ldots,a_{1(m_1)},k)}(X)}{(X)_k}\\
                            & \hspace{10mm} -a_{1(1)}\frac{(X)_{a_{1(1)}+a_{1(2)}}P_{T(1,\bar{S_1},\ldots,S_n,k)}(X)}{(X)_{a_{1(2)}+1}}.\\
\end{align*}
\end{lem}
\begin{proof}
This follows from a simple application of the deletion-contraction identity.  Let $v$ be the singleton endpoint verteX of the clique-theta graph.  Deleting the edges between $v$ and the $a_{1(1)}$-clique produces a $K_k$-sum of $T(1,S_2,\ldots,S_n,k)$ and $L(a_{1(1)},\ldots,a_{1(m_1)},k)$.  Contracting one of these edges turns the graph into a $K_{a_{1(2)}+1}$-sum of $T(1,\bar{S_1},\ldots,S_n,k)$ and $K_{a_{1(1)}+a_{1(2)}}$.  As the contraction of any of these edges produces the same graph, the chromatic polynomial of the latter appears with multiplicity $a_{1(1)}$.
\end{proof}

We can now proceed by induction on the sizes of the sets $S_i$.

\begin{proof}[Proof of Proposition \ref{prop:cliquetheta}]
First suppose that the size of each set is 1; that is, suppose $S_i=(a_{i(1)})$ for all $i$.  Let $v$ be the singleton extremal vertex, which is in this case connected to all other vertices of the graph apart from those of the $k$-clique.  Note that contracting any added edge between $v$ and the $k$-clique produces a $K_k$-sum of $(a_{i(1)}+k)$-cliques.  This graph has chromatic polynomial:
\[f(X)=\frac{\prod_{i=1}^n(X)_{a_{i(1)+k}}}{(X)_k^{n-1}}.\]
Also note that adding all edges between $v$ and the $k$-clique gives a $K_{k+1}$-sum of $(a_{i(1)}+k+1)$-cliques, having chromatic polynomial:
\[g(X)=\frac{\prod_{i=1}^n(X)_{a_{i(1)+k+1}}}{(X)_{k+1}^{n-1}}.\]
We now apply the contraction-addition identity $k$ times, where ``addition'' consists of adding an edge between $v$ and the $k$-clique, and ``contraction'' consolidates the two vertices between which the new edge is to be added.  At every stage the consolidation of these two vertices will produce the graph with chromatic polynomial $f(X)$, and so our final graph will have chromatic polynomial which is a sum of $g(X)$ with $k$ copies of $f(X)$, taking the following form:
\begin{align*}
                           &\left(k(X)_{a_{n(1)}+k}\prod_{i=1}^{n-1}(X-k)_{a_{i(1)}}\right)+\left((X)_{a_{n(1)}+k+1}\prod_{i=1}^{n-1}(X-k-1)_{a_{i(1)}}\right)\\
                           &=\left((X)_{a_{n(1)}+k}\prod_{i=1}^{n-1}(X-k-1)_{a_{i(1)}-1}\right)\left(k(X-k)^{n-1}+\prod_{i=1}^n(X-a_{i(1)}-k)\right).\\
\end{align*}

Note that $r(1,a_{i(1)};X)=1$ and $r(1,a_{i(1)},k;X)=X-a_{i(1)}-k$ for all $i$.  Hence Proposition \ref{prop:cliquetheta} holds when $|S_i|=1$ for all $i$.

These graphs suffice as the base case for the induction, as we can build up any clique-theta graph by starting with one having $|S_i|=1$ for all $i$, and systematically increasing the length of the clique-paths.  Reordering the $S_i$ does not alter the graph, so for ease of notation we can assume that at each stage the clique-path to which we are adding a new clique is $L(1,S_1,k)$.  In a similar way, at each stage we can shift the labelling of the individual cliques up one, so that the new element being added to $S_1$ is always labelled $a_{1(1)}$.

\begin{figure}[hbtp]
\begin{minipage}[b]{0.4\linewidth}
\centering
\begin{tikzpicture}
 [scale=.4,auto=center,every node/.style={circle,fill=black!100,inner sep=1mm},
    every path/.style={draw,thick,black}]
  \node (n1a) at (-5,0.5) {};
  \node (n2a) at (-9.5,4) {};
  \node (n5a) at (-9.5,10) {};
  \node (n7a) at (-5,13.5) {};
  \node (n10a) at (-5.7,7) {};
  \node (n12a) at (-1.2,2) {};
  \node (n13a) at (.3,7) {};
  \node (n15a) at (-1.2,11) {};

  \foreach \from/\to in {n1a/n10a,n10a/n7a}
  \path[bend left=3,thick] (\from) edge (\to);
    \path[bend left=18,thick] (n1a) edge (n2a);
    \path[bend left=20,thick] (n2a) edge (n5a);
    \path[bend left=18,thick] (n5a) edge (n7a);
    \foreach \from/\to in {n1a/n12a,n12a/n13a,n13a/n15a,n15a/n7a}
    \path[bend right=12,thick] (\from) edge (\to);
   \end{tikzpicture}
\end{minipage}
\begin{minipage}[t]{0.1\linewidth}
\centering
\begin{tikzpicture}[scale=.3,auto=center,every node/.style={circle,fill=black!100,inner sep=1mm},
    every path/.style={draw,thick,black}]
    \begin{scope}[yshift=50mm]
    \path[bend left=18,dotted,thick] (0,10) edge[white,->] (2,10);
    \path[bend right=18,dotted,thick] (0,-3) edge[white,->] (2,-3);
    \path[dotted,thick] (0,4) edge[->] (2,4);
\end{scope}
\end{tikzpicture}

\end{minipage}
\begin{minipage}[t]{0.4\linewidth}
\centering
\begin{tikzpicture}
    [scale=.4,auto=center,every node/.style={circle,fill=black!100,inner sep=1mm},
    every path/.style={draw,thick,black}]
  \node (n1) at (10,0) {};
  \node (n2) at (4.3,3.6) {};
  \node (n3) at (6.2,4) {};
  \node (n4) at (4.3,10.4) {};
   \node (n5) at (6.2,10) {};
  \node (n7) at (8.5,14) {};
  \node (n8) at (11.5,14) {};
  \node (n9) at (10,6) {};
  \node (n10) at (8.5,8) {};
  \node (n11) at (11.5,8) {};
  \node (n12) at (14,2) {};
  \node (n13) at (14.7,7) {};
  \node (n14) at (16.4,7.2) {};
  \node (n15) at (14,11) {};

  \foreach \from/\to in {n2/n3,n4/n5,n1/n9,n9/n10,n9/n11,n10/n11,n10/n7,n11/n8,n13/n14}
    \draw (\from) -- (\to);

   \foreach \from/\to in {n1/n2,n1/n3,n2/n4,n2/n5,n3/n4,n3/n5,n4/n7,n5/n7,n4/n8,n5/n8,n7/n8,n1/n10,n10/n8}
   \path[bend left=10,thick] (\from) edge (\to);
   \foreach \from/\to in {n1/n12,n12/n14,n14/n15,n15/n7,n15/n8,n1/n11,n11/n7}
   \path[bend right=10,thick] (\from) edge (\to);
   \foreach \from/\to in {n12/n13,n13/n15,n9/n7}
   \path[bend right=5,thick] (\from) edge (\to);
   \path[bend left=5,thick] (n9) edge (n8);
   \end{tikzpicture}
\end{minipage}
\caption{The clique-theta graph $T(1,(2,2),(3),(1,2,1),2)$, and its underlying generalised theta graph $\Theta_{4,3,5}$}
\end{figure}
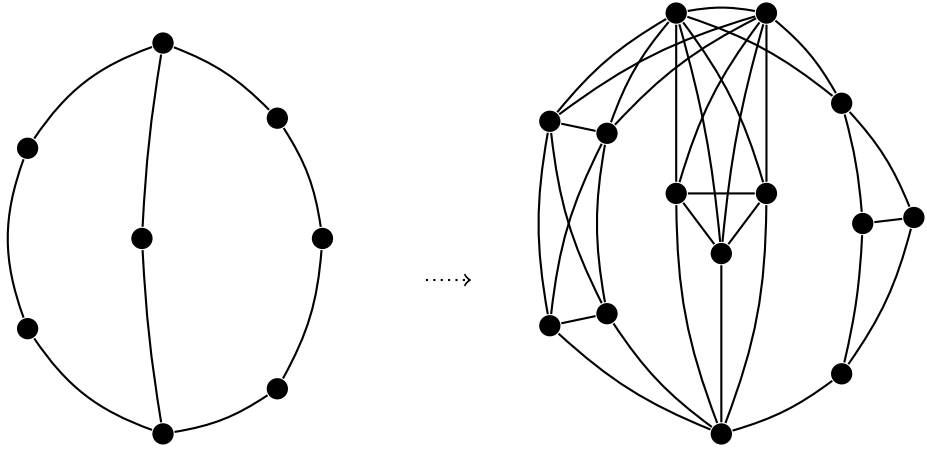

Thus, by Lemma \ref{lem:thetarecur}, we need only show that if Proposition \ref{prop:cliquetheta} holds for $T(1,S_2,\ldots,S_n,k)$ and $T(1,\bar{S_1},\ldots,S_n,k)$ then it holds too for $T(1,S_1,\ldots,S_n,k)$.  So let us assume that $T(1,S_2,\ldots,S_n,k)$ and $T(1,\bar{S_1},\ldots,S_n,k)$ have chromatic polynomials of the stated form, and define:
\[f(X)=(X)_{a_{n(m_n)}+k}\left(\prod_{i=1}^{n-1}(X-k-1)_{a_{i(m_i)}-1}\right)\left(\prod_{i=1}^n\prod_{l=1}^{m_i-1}(X-a_{i(l+1)}-1)_{a_{i(l)}-1}\right).\]

Then removing $f(X)$ as a factor from the expressions
\[\frac{1}{{(X)_k}}\left(P_{T(1,S_2,\ldots,S_n,k)}(X)P_{L(a_{1(1)},\ldots,a_{1(m_1)},k)}(X)\right)\]
and
\[\frac{1}{(X)_{a_{1(2)}+1}}\left((X)_{a_{1(1)}+a_{1(2)}}P_{T(1,\bar{S_1},\ldots,S_n,k)}(X)\right)\]
leaves us with, respectively:
\begin{equation}\label{eqn:1}
  \begin{aligned}
\prod_{l=2}^{m_1} &(X-a_{1(l)})k(X-k)^{n-1}\prod_{i=2}^nr(1,a_{i(1)},\ldots, a_{i(m_i)};X)\\
 & +(X-k)\left(\prod_{l=2}^{m_1}(X-a_{1(l)})\right)\left(\prod_{i=2}^nr(1,a_{i(1)},\ldots,a_{i(m_i)},k;X)\right),
\end{aligned}
\end{equation}
and
\begin{equation}\label{eqn:2}
  \begin{aligned}
k(X-k)^{n-1}&r(1,a_{1(2)},\ldots a_{1(m_1)};X)\left(\prod_{i=2}^nr(1,a_{i(1)},\ldots a_{i(m_i)};X)\right)\\
&+r(1,a_{1(2)},\ldots a_{1(m_1)},k)\left(\prod_{i=2}^nr(1,a_{i(1)},\ldots a_{i(m_i)},k;X)\right).
\end{aligned}
\end{equation}

Now, we define the interesting factor of the chromatic polynomial of a general clique-theta graph in much the same as we did for rings of cliques, to be that factor which is left upon dividing out all linear factors.  By Lemma \ref{lem:thetarecur}, the interesting factor of the chromatic polynomial of $T(1,S_1,\ldots,S_n,k)$ is obtained by subtracting (\ref{eqn:2}) $a_{1(1)}$ times from (\ref{eqn:1}).  This gives:
\begin{align*}
&\left(k(X-k)^{n-1}\prod_{i=2}^nr(1,a_{i(1)},\ldots a_{i(m_i)})\right)\left[\left(\prod_{l=2}^{m_1}(X-a_{1(l)})\right)-a_{1(1)}r(1,a_{1(2)},\ldots a_{1(m_1)})\right]\\
&+\left(\prod_{i=2}^nr(1,a_{i(1)},\ldots,a_{i(m_i)},k)\right)\left[(X-k)\left(\prod_{l=2}^{m_1}(X-a_{1(l)})\right)-a_{1(1)}r(1,a_{1(2)},\ldots a_{1(m_1)},k)\right],
\end{align*}
(note that we have removed the $X$'s from our notation for the interesting factors of rings of cliques here; this is simply so as to improve presentation).

To complete the proof it suffices simply to note that
\[\left(\prod_{l=2}^{m_1}(X-a_{1(l)})\right)-a_{1(1)}r(1,a_{1(2)},\ldots a_{1(m_1)};X)=r(1,a_{1(a)},\ldots a_{1(m_1)};X).\]
\end{proof}

\section{Proof of the main result}

We are now in a position to be able to prove the main theorem of this paper.  First of all, Proposition \ref{prop:main} is obtained as a corollary of the following result.
\begin{prop}\label{prop:palpha}
Let $S_i=(a_{i(1)},\ldots,a_{i(m)})$, let $pS_i=(pa_{i(1)},\ldots,pa_{i(m)}),$ and let $\alpha$ be a non-integer chromatic root of the clique-theta graph $T(1,S_1,\ldots,S_n,k)$.  Then for any natural number $p$, the product $p\alpha$ is a chromatic root of $T(1,pS_1,\ldots,pS_n,pk)$.
\end{prop}
\begin{proof}
We need only consider the non-linear factor of the chromatic polynomial.  For $T(1,S_1,\ldots,S_n,k)$ this is, by Proposition \ref{prop:cliquetheta}:
\begin{equation}\label{eqn:thetaalpha}
\left[k(X-k)^{n-1}\prod_{i=1}^nr(1,a_{i(1)},\ldots, a_{i(m_i)};X)\right]+\left[\prod_{i=1}^nr(1,a_{i(1)},\ldots a_{i(m_i)},k;X)\right].
\end{equation}

Expanding the interesting factors of the rings of cliques, this becomes
\begin{align}\label{eqn:thetaalpha}
&\left[k(X-k)^{n-1}\prod_{i=1}^n\frac{1}{X}\left(\prod_{l=1}^{m_i}(X-a_{i(l)})-\prod_{l=1}^{m_i}(-a_{i(l)})\right)\right] \nonumber \\
& \hspace{30mm} +\left[\prod_{i=1}^n\frac{1}{X}\left((X-k)\prod_{l=1}^{m_i}(X-a_{i(l)})+k\prod_{l=1}^{m_i}(-a_{i(l)})\right)\right].
\end{align}

For $T(1,pS_1,\ldots,pS_n,pk)$, we have
\begin{align}\label{eqn:thetapalpha1}
&\left[pk(X-pk)^{n-1}\prod_{i=1}^n\frac{1}{X}\left(\prod_{l=1}^{m_i}(X-pa_{i(l)})-\prod_{l=1}^{m_i}(-pa_{i(l)})\right)\right] \nonumber \\
& \hspace{30mm} +\left[\prod_{i=1}^n\frac{1}{X}\left((X-pk)\prod_{l=1}^{m_i}(X-pa_{i(l)})+pk\prod_{l=1}^{m_i}(-pa_{i(l)})\right)\right].
\end{align}

Let $s=\sum_{i=1}^n(m_i+1).$  Then dividing (\ref{eqn:thetapalpha1}) by $p^s$ produces
\begin{align}\label{eqn:thetapalpha2}
&\left[k(X/p-k)^{n-1}\prod_{i=1}^n\frac{1}{X}\left(\prod_{l=1}^{m_i}(X/p-a_{i(l)})-\prod_{l=1}^{m_i}(-a_{i(l)})\right)\right] \nonumber \\
& \hspace{30mm} +\left[\prod_{i=1}^n\frac{1}{X}\left((X/p-k)\prod_{l=1}^{m_i}(X/p-a_{i(l)})+k\prod_{l=1}^{m_i}(-a_{i(l)})\right)\right],
\end{align}

and if $\alpha$ is a zero of (\ref{eqn:thetaalpha}), then $p\alpha$ is a zero of (\ref{eqn:thetapalpha2}).
\end{proof}

Proposition \ref{prop:main} can now be verified as follows.

\begin{proof}[Proof of Proposition \ref{prop:main}]
Let $\Theta_{m_1,\ldots,m_n}$ be a generalised theta graph having a non-integer chromatic root $\alpha$.  If, for each $1\leq i\leq n$, we let $S_i$ be the ordered $(m_i-2)$-tuple $(1,1,\ldots,1)$, then we may express $\Theta_{m_1,\ldots,m_n}$ instead as $T(1,S_1,\ldots,S_n,1).$  Thus $\alpha$ is a chromatic root of $T(1,S_1,\ldots,S_n,1),$ and so by Proposition \ref{prop:palpha}, $p\alpha$ is a chromatic root of $T(1,pS_1,\ldots,pS_n,p)$ for any $p\in \NNN.$
\end{proof}

There still remains some work to do in order to prove the existence of a set satisfying the properties claimed in the statement of Theorem \ref{thm:main}.  Sokal showed in \cite{Sokal:Dense} that the chromatic root of generalised theta graphs are dense in the whole complex plane \emph{wih the exception of the disc} $\{z\in\CCC: |z-1|<1\}$.  He was able to achieve the full result by simply observing that the chromatic roots of any graph can be ``shifted up'' by 2 simply by joining the graph in question with a copy of $K_2$.

We have a slightly more difficult task in filling this hole, as we must do something similar, but in such a way as to preserve the all-important ``closure under positive integer multiplication'' property.   As it turns out, Sokal's simple fix can be adapted in a quite pleasing way.  The final part of our verification of Theorem \ref{thm:main} is a direct proof for the excluded disc in question.
\begin{proof}[Proof for the disc $|z-1|<1$, and thus of Theorem \ref{thm:main}]
Let $z \in \mathbb{C}$ be such that $|z-1|<1$, and let $\epsilon >0$ be arbitrarily small (in particular we may assume that $\epsilon < |2-\Re(z)|$).  We wish to show that there is a chromatic root $\alpha$ which satisfies the following two conditions:
\begin{enumerate}[label=\emph{(\alph*)}]
\item $|\alpha - z|< \epsilon$
\item  For all natural numbers $p$ there exists a graph having chromatic root $p\alpha$
\end{enumerate}

Let $w=z-2$.  Then, by Sokal's result, there exists some generalised theta graph $G_1$ having a chromatic root $\beta$ such that $|\beta - w|<\epsilon$.  Now let $G_2$ be the join of $G_1$ with $K_2$, and let $\alpha = \beta+2$.  Then $\alpha$ is a chromatic root of $G_2$ by Proposition \ref{prop:basic} \emph{(i)}, and we have:
\[|\alpha - z| = |\beta - w| < \epsilon,\]
thus proving part \emph{(a)}.

For part \emph{(b)}, let $n$ be any natural number, and let $G_n$ be the clique-theta graph obtained from $G_1$ by blowing up all but one endpoint vertex into a clique of size $n$.  By Proposition \ref{prop:main}, $G_n$ then has a chromatic root $n\beta$, and so the join of $G_n$ with $K_{2n}$ has a chromatic root:
\[n\beta + 2n = n(\beta+2) = n\alpha\]

Finally we can explicitly construct a set whose existence is claimed in Theorem \ref{thm:main}.  Let $S$ be the set containing all clique-theta graphs (with at least one trivial extremal clique), along with every possible join of these graphs with a complete graph of even order.  Then the chromatic roots of the elements of $S$ satisfy the two conditions of Theorem \ref{thm:main}.
\end{proof}

\section{Concluding remarks}
It is natural to ask whether or not the rare direct correspondence between graph structure and chromatic roots exploited in this paper might generalise to a larger class of graphs, or hold for an unrelated one.  Our searches have proved fruitless in this respect; in particular, computations with series-parallel graphs and outerplanar graphs---the most natural families to try, due to similarities in structure and size---have turned up not a single example of a different type of graph whose chromatic roots increase in size by the same factor as its blown-up vertex-cliques. It is tempting to speculate that this property may be unique to these graphs; it would be interesting to investigate if this is so, and why.

We would also draw attention to the potential of the formula derived for the chromatic polynomials of clique-theta graphs for computations.  Clique-theta graphs are now---to the best of our knowledge, and in a loose sense---the largest family of graphs for which a single parametrised chromatic polynomial formula is known.  They generalise a number of graphs with interesting chromatic properties, and combine their usefulness for computations.  For example, rings of cliques have been of much value in algebraic investigation of the chromatic polynomial due to the fact that, once the path-lengths are chosen, varying the parameters setting clique size does not change the degree of the polynomial.  On the other hand, the other well-studied specialisations---generalised theta graphs---have chromatic polynomials containing variable parameters as exponents, enabling the study of chromatic roots of arbitrarily large absolute value as in Sokal's work.  These two attributes combine in the formula we provide here, and we hope it will prove useful as a large and varied source of chromatic polynomials.

\subsection*{Acknowledgements}
This paper was largely written while under the supervision of Peter Cameron at Queen Mary, University of London.  I would like to thank Prof. Cameron for suggesting this field of research, and for his helpful comments and advice.  My thanks also go to the EPSRC, without whose financial support this work would not have been possible.  Finally, I am much obliged to Bill Jackson for pointing out that my original proof of Theorem \ref{thm:main} did not cover the disc $|z-1|<1$, and to an anonymous referee, who also spotted this omission, and gave various other suggestions which led to a great improvement in the presentation of this paper.

\bibliographystyle{plain}
\bibliography{CliqueThetas2}
\end{document}